\documentclass[preprint,12pt,times]{elsarticle}

\usepackage{amssymb}
\usepackage{amsmath}
\usepackage{amsthm}
\usepackage{geometry}
\usepackage{setspace}
\usepackage{color}
\usepackage[
  colorlinks,
  linkcolor=blue,
  anchorcolor=blue,
  citecolor=blue,
  urlcolor=blue
]{hyperref}

\newtheorem{theorem}{Theorem}[section]
\newtheorem{corollary}[theorem]{Corollary}
\newtheorem{lemma}[theorem]{Lemma}

\theoremstyle{definition}

\theoremstyle{remark}

\numberwithin{equation}{section}
\allowdisplaybreaks
\newcommand{\Disk}{\mathbb D}
\newcommand{\Torus}{\mathbb T}

\newcommand{\AL}{AL}
\newcommand{\AT}{AT}
\newcommand{\Ces}{\mathcal C}
\newcommand{\Dclass}{\mathcal D}
\newcommand{\Dhat}{\widehat{\mathcal D}}
\newcommand{\Dcheck}{\check{\mathcal D}}
\newcommand{\what}[1]{\widehat{#1}}
\newcommand{\Ctail}{\mathfrak C_{\beta,p}}

\newcommand{\Tcut}{T_{\mu,\beta;n_0}^{K}}

\newcommand{\stmtrev}[1]{#1}
\newcommand{\roundrev}[1]{#1}

\newcommand{\red}[1]{#1}
\begin{document}

\begin{frontmatter}

\title{Ces\`aro-type operators on average radial integrability spaces induced by doubling weights}

\author[aff1]{Jiaxin Pan}
\ead{jxpan0209@163.com}

\author[aff1]{Cezhong Tong}
\ead{ctong@hebut.edu.cn}
\ead{cezhongtong@hotmail.com}
\author[aff1]{Zixing Yuan\texorpdfstring{\corref{cor1}}{}}
\ead{yuan980127@163.com}
\cortext[cor1]{Corresponding author}
\affiliation[aff1]{Institute of Mathematics, Hebei University of Technology,
		Tianjin 300401, China}

\begin{abstract}
In this paper, we study the Ces\`aro-type operator
$$
\Ces_{\mu,\beta}f(z)=\int_{[0,1)}\frac{f(tz)}{(1-tz)^\beta}\,d\mu(t),
$$
where \(\mu\) is a positive Borel measure on \([0,1)\) and \(\beta>0\).
For \(0<p,q<\infty\) and radial doubling weights
\(\omega_1,\omega_2\), we prove that
\(\Ces_{\mu,\beta}:\AL_p^q(\omega_1)\to\AL_p^q(\omega_2)\) is bounded if
and only if
$$
\mu([0,1))+
\sup_{1/2\le r<1}\frac{\mu([r,1))}{(1-r)^\beta}
\left(\frac{\what\omega_2(r)}{\what\omega_1(r)}\right)^{1/p}<\infty,
$$
and its norm is comparable to this quantity. For standard weights, our result extends
\cite{BlascoMas2026,GalanopoulosSiskakisZhao2025} to the full range
$0<p=q<\infty$. We further characterize compactness and the essential norm
 on average radial integrability spaces
induced by doubling weights.
\end{abstract}


\end{frontmatter}

\thispagestyle{plain}
\section{Introduction}\label{sec:introduction}

 Let $H(\mathbb D)$ denote the space of analytic functions in the unit disc $\mathbb{D}=\{z \in \mathbb C:|z|<1\}$. For \(0<p,q<\infty\) and a radial weight
\(\omega\), let \(L_p^q(\omega)\) be the space of complex-valued measurable functions $f$ on $\mathbb{D}$ such that
$$
 \|f\|_{L_p^q(\omega)}
 =\left(\int_0^{2\pi}
 \left[\int_0^1|f(re^{i\theta})|^p\omega(r)r\,dr\right]^{q/p}
 \frac{d\theta}{2\pi}\right)^{1/q}<\infty.
$$
The weighted average radial integrability space is
$$
 \AL_p^q(\omega)=L_p^q(\omega)\cap H(\mathbb D).
$$
Taking $\omega\equiv1$ yields the unweighted case.
The unweighted average radial integrability spaces were introduced in
\cite{AguilarHernandezContrerasRodriguezPiazza2022}, and further operator theory of
these spaces can be found in
\cite{AguilarHernandezContrerasRodriguezPiazza2021,AguilarHernandezGalanopoulos2023}.
The case of general radial weights was studied in
\cite{AguilarHernandezEtAl2026}, where several basic properties, including
point evaluation estimates, Littlewood-Paley formulas, maximal function
estimates, and boundedness of the Bergman projection, etc., were established. For the standard weight
$\omega_\alpha(r)=(\alpha+1)(1-r^2)^\alpha, \alpha>-1$, and
$0<p=q<\infty$, Fubini's theorem gives
$
 \AL_p^p(\omega_\alpha)=A_\alpha^p
$
with equivalent quasinorms. Thus, the classical weighted Bergman spaces
\(A_\alpha^p\) are special cases of \(\AL_p^q(\omega)\).

For \(\beta>0\) and let \(\mu\) be a positive Borel measure on \([0,1)\). We consider the Ces\`aro-type operator
\begin{equation}\label{eq:intro-operator}
	\Ces_{\mu,\beta}f(z)
	=\int_{[0,1)}\frac{f(tz)}{(1-tz)^\beta}\,d\mu(t),
	\qquad z\in\Disk.
\end{equation}
When \(d\mu(t)=dt\) and \(\beta=1\), \eqref{eq:intro-operator} is the
classical Ces\`aro operator. The operator properties of the classical Ces\`aro operators were first studied by Brown, Halmos and Shields \cite{BHS} in 1965.  Since then, many authors have investigated this operator, see for example \cite{Andersen,ShiRen1998,Siskakis1987,Siskakis1990,Siskakis1996}. In this paper, we focus on the
measure-induced Ces\`aro-type operator \eqref{eq:intro-operator}.

Ces\`aro-type operators induced by measures have been studied intensively in the past few decades. Galanopoulos, Girela, and Merch\'an
\cite{GalanopoulosGirelaMerchan2022} investigated Ces\`aro-type operators on several classical spaces of analytic functions, including weighted Bergman spaces,
Hardy spaces, \(BMOA\), and the Bloch space. A correction to the
Bergman space sufficiency argument appeared later in
\cite{GalanopoulosGirelaMerchanCorrection2026}. Bao, Sun, and Wulan
\cite{BaoSunWulan2022} studied  
Ces\`aro-type operator on the space of bounded analytic functions. 
Recently, Galanopoulos,
Siskakis, and Zhao \cite{GalanopoulosSiskakisZhao2025} studied the boundedness of  Ces\`aro-type operators from $A_\alpha^p$ to $A_\alpha^q$ for $1\le p\le q<\infty$. 
Blasco and Mas \cite{BlascoMas2026} considered Ces\`aro-type
operators on the classical mixed norm spaces \(H(p,q,\gamma)\). Recall that for $0<p,q,\gamma<\infty,$ \(H(p,q,\gamma)\) consists of  analytic functions $f$ on $\mathbb D$ such that
\[
 \|f\|_{H(p,q,\gamma)}
 =\left(\int_0^1M_p(f,r)^q(1-r)^{\gamma q-1}\,dr\right)^{1/q},
\]
where \(M_p(f,r)=\left(\int_0^{2\pi}|f(re^{i\theta})|^p\frac{d\theta}{2\pi}\right)^{1/p}\).
We reminder the readers that \(H(p,q,\gamma)\) and \(\AL_p^q(\omega_\alpha)\) are different.
 For \(0<p=q<\infty\), however, Fubini's
theorem yields
$
\AL_p^p(\omega_\alpha)=A_\alpha^p
=H\!\left(p,p,(\alpha+1)/p\right)
$
with equivalent quasinorms. 
Related measure-induced Ces\`aro-type operators on Hilbert spaces, derivative-type
Hilbert spaces
and Dirichlet-type spaces were studied in
\cite{GalanopoulosGirelaMasMerchan2023,GalanopoulosGirelaMerchan2023,JinTang2022,LinXie2025,XieLiuLinCorrection2026,XieLiuLin2026}.

Motivated by the results above, we completely characterize the boundedness of \(\Ces_{\mu,\beta}\) between weighted average radial integrability spaces with doubling weights for all \(0<p,q<\infty\). In the
standard weight case with \(p=q\), Theorem~6.10 of \cite{BlascoMas2026} characterizes the
boundedness of \(\Ces_{\mu,\beta}\) from \(H\!\left(p,p,(\alpha_1+1)/p\right)\) to \( H\!\left(p,p,(\alpha_2+1)/p\right)\) for \(p\ge1\) under
\(s=\beta+\frac{\alpha_1-\alpha_2}{p}>0\). For \(0<p<1\), Lemmas~6.8 and~6.9 of that paper give the boundedness
characterization when \(s>0\) and
\(\alpha_2>\alpha_1\), or equivalently, when \(s>0\) and \(\beta>s\).
This partially extends the results of Galanopoulos et al.~\cite{GalanopoulosSiskakisZhao2025}
to the case \(0<p<1\).  Our results removes the additional condition  \(\alpha_2>\alpha_1\) and give a complete characterization for $0<p<1$, thereby filling this gap. Furthermore, the case  $s\leq 0$ is also studied in our theorem.
When \(p\ne q\), the boundedness results in
\cite{BlascoMas2026,GalanopoulosSiskakisZhao2025} do not directly imply our
boundedness theorem even for standard weights because of the different order
of integration. Beyond boundedness, we also characterize
compactness and the essential norm in the two-weight setting.

We now state our main results. For a radial weight \(\omega\), write
$
 \what\omega(r)=\int_r^1\omega(s)\,ds.
$
We shall work with the class of two-sided doubling weights
\(\Dclass=\Dhat\cap\Dcheck\), defined in Section~\ref{sec:preliminaries}.
Our first main result gives the boundedness characterization in the
two-weight setting.

\begin{theorem}\label{thm:main-two-weight}
For \(0<p,q<\infty\), \(\beta>0\),
 \(\omega_1,\omega_2\in\Dclass\), and let \(\mu\) be a positive Borel
 measure on \([0,1)\). Define
 \begin{equation}\label{eq:intro-C}
 \begin{split}
 \Ctail(\mu;\omega_1,\omega_2)
 :=\;&\mu([0,1))+\sup_{1/2\le r<1}
 \frac{\mu([r,1))}{(1-r)^\beta}
 \left(\frac{\what\omega_2(r)}{\what\omega_1(r)}\right)^{1/p}.
 \end{split}
 \end{equation}
 Then
$
 \Ces_{\mu,\beta}:\AL_p^q(\omega_1)\rightarrow
 \AL_p^q(\omega_2)
 $
 is bounded if and only if
 \(\Ctail(\mu;\omega_1,\omega_2)<\infty\). Moreover,
 \begin{equation*}
 \bigl\|\Ces_{\mu,\beta}\bigr\|_{
 \AL_p^q(\omega_1)\to\AL_p^q(\omega_2)}
 \asymp
 \Ctail(\mu;\omega_1,\omega_2).
 \end{equation*}
\end{theorem}

For a bounded operator \(T:X\to Y\), its essential norm is
\(\|T\|_e=\inf\{\|T-K\|:K:X\to Y\text{ is compact}\}\). Our second
result \red{concerns} the essential norm and compactness.

\begin{theorem}\label{thm:intro-essential-norm}
For \(0<p,q<\infty\), \(\beta>0\),
 \(\omega_1,\omega_2\in\Dclass\), \stmtrev{and let \(\mu\) be a positive Borel
 measure on \([0,1)\).} Suppose that
 \(\Ces_{\mu,\beta}:\AL_p^q(\omega_1)\to\AL_p^q(\omega_2)\) is bounded.
 Then
 \begin{equation}\label{eq:intro-essential-norm}
 \bigl\|\Ces_{\mu,\beta}\bigr\|_e
 \asymp
 \limsup_{r\to1^-}
 \frac{\mu([r,1))}{(1-r)^\beta}
 \left(\frac{\what\omega_2(r)}{\what\omega_1(r)}\right)^{1/p}.
 \end{equation}
 Consequently, \(\Ces_{\mu,\beta}\) is compact if and only if
 \begin{equation}\label{eq:intro-compactness-condition}
 \lim_{r\to1^-}
 \frac{\mu([r,1))}{(1-r)^\beta}
 \left(\frac{\what\omega_2(r)}{\what\omega_1(r)}\right)^{1/p}=0.
 \end{equation}
\end{theorem}

The paper is organized as follows. Section~\ref{sec:preliminaries} collects
the preliminary lemmas. Section~\ref{sec:main-theorem} establishes the boundedness result by controlling the Ces\`aro-type operator  through a shifted sequence operator  acting on weighted sequence spaces.
 Section~\ref{sec:essential-norm}
studies compactness and the essential norm. Three applications of our main results are given in
Section~\ref{sec:examples-scope}.

We write \(A\lesssim B\) if \(A\le CB\) for some constant \(C>0\), and
\(A\asymp B\) if both \(A\lesssim B\) and \(B\lesssim A\) hold.
\section{Preliminaries}\label{sec:preliminaries}
We begin by collecting the basic properties of the radial doubling weights and several
lemmas needed in the proofs.

\subsection{Radial doubling weights}\label{subsec:doubling-weights}

We first recall the radial doubling weights used throughout the
paper. These doubling conditions arise naturally in the theory of weighted
Bergman spaces and related operators. These classes were introduced and
systematically studied by Pel\'aez and R\"atty\"a, see for example
\cite{PelaezSmall2016,PelaezRattya2014,PelaezRattya2015,
PelaezRattya2016,PelaezRattya2021}.
A radial weight is a nonnegative integrable function on \([0,1)\),
extended to \(\Disk\) by \(\omega(z)=\omega(|z|)\). We assume that its tail
function \(\what\omega(r)=\int_r^1\omega(s)\,ds\) is positive for every
\(r<1\).

A radial weight \(\omega\) belongs to \(\Dhat\) if there exists \(A>1\) such
that
\begin{equation}\label{eq:D-plus}
 \what\omega(r)\le A\what\omega\!\left(\frac{1+r}{2}\right),
 \qquad 0\le r<1.
\end{equation}
\roundrev{We say that
\(\omega\in\Dcheck\) if there exist constants \(B>1\), \(\kappa>1\), and
\(r_0<1\) such that}
\begin{equation}\label{eq:D-minus}
 \what\omega(r)\ge B\what\omega\!\left(1-\frac{1-r}{\kappa}\right),
 \qquad r_0\le r<1.
\end{equation}
\roundrev{Finally, we write \(\Dclass=\Dhat\cap\Dcheck\). Thus, weights in
\(\Dclass\) satisfy both the doubling and reverse doubling conditions and will
be referred to as two-sided doubling weights.} Basic properties of these classes can be found in
\cite{PelaezSmall2016,PelaezRattya2014}.
For \(\alpha>-1\), the standard weight
\(\omega_\alpha(r)=(\alpha+1)(1-r^2)^\alpha \in \Dclass\) with
$
 \what\omega_\alpha(r)\asymp(1-r)^{\alpha+1}.
$

We recall several results from \cite{AguilarHernandezEtAl2026} that will be used in the proof.
For a measurable function \(f:\mathbb D\to\mathbb C\), define its radial
maximal function by
\[
 R(f)(z)=\sup_{0\le\rho\le1}|f(\rho z)|,\qquad z\in\mathbb D.
\]
The maximal function estimate can be found in
\cite[Theorem~2.1]{AguilarHernandezEtAl2026}.

\begin{lemma}\label{lem:cited-maximal}
 For \(0<p,q<\infty\) and let \(\omega\) be a radial weight. Then
 \begin{equation*}
 \|R(f)\|_{L_p^q(\omega)}
 \lesssim_{p,q}\|f\|_{\AL_p^q(\omega)},
 \qquad f\in H(\mathbb D).
 \end{equation*}
\end{lemma}

The following point evaluation estimate follows from
\cite[Lemma~3.1]{AguilarHernandezEtAl2026}.
\begin{lemma}\label{lem:cited-point-evaluation}
 For \(0<p,q<\infty\) and \(\omega\in\Dhat\). Then, for every
 \(f\in\AL_p^q(\omega)\) and \(a\in\Disk\),
 \begin{equation*}
 |f(a)|\lesssim
 \frac{\|f\|_{\AL_p^q(\omega)}}
 {\what\omega(|a|)^{1/p}(1-|a|)^{1/q}}.
 \end{equation*}
\end{lemma}

The following result was proved in 
\cite[Lemma~4.1]{AguilarHernandezEtAl2026}.
\begin{lemma}\label{lem:cited-high-kernel}
For \(0<p,q<\infty\) and \(\omega\in\Dhat\). There exists
 \(\beta_0=\beta_0(\omega)>0\) such that, if
 \(\gamma>\frac1q+\frac{\beta_0}{p}\), then for every \(a\in\mathbb D\), the function
 \[
 f_a(z)=\frac{1}{(1-\overline a z)^\gamma},
 \qquad z\in\Disk,
 \]
 satisfies
 \begin{equation*}
 \|f_a\|_{\AL_p^q(\omega)}
 \lesssim
 \what\omega(|a|)^{1/p}(1-|a|)^{1/q-\gamma}.
 \end{equation*}
\end{lemma}

\subsection{A shifted sequence operator}\label{subsec:main-tools}

In this subsection, we obtain an upper estimate for
\(\Ces_{\mu,\beta}\) in terms of a shifted sequence operator, which will be a
key tool in the proof of our main results. Fix \(K>1\), and set
\begin{equation*}
	r_n=1-K^{-n},\qquad I_n=[r_n,r_{n+1}),\qquad n\ge0.
\end{equation*}
For a positive Borel measure \(\mu\) and radial weights \(\omega_i\), \(i=1,2\), let
\begin{equation*}
 m_n=\mu(I_n),\qquad M_n=\mu([r_n,1)),\qquad
 w_{i,n}=\int_{I_n}\omega_i(r)r\,dr.
\end{equation*}
Then, we have
$
 M_n=m_n+M_{n+1}.
$
For a sequence \(a=(a_n)_{n\ge0}\), define
\begin{equation}\label{eq:shifted-discrete-operator}
 (T_{\mu,\beta}^{K}a)_n
 :=K^{\beta n}M_{n+1}a_{n+1}
 +\sum_{k=0}^nK^{\beta k}m_ka_{k+1}.
\end{equation}

The next result connects the integral operator \(\Ces_{\mu,\beta}\) with the
shifted discrete operator defined above. It gives a pointwise bound and the corresponding mixed norm estimate for \(\Ces_{\mu,\beta}f\).

\begin{lemma}\label{thm:shifted-domination}
 For \(0<p,q<\infty\), \(\beta>0\), \stmtrev{\(K>1\), and let \(\mu\) be a
 positive Borel measure on \([0,1)\).} \roundrev{Let \(\omega_2\) be a radial
 weight. Then, for every \(f\in H(\mathbb D)\), \(n\ge0\), \(r\in I_n\), and
 \(\theta\in\mathbb R\), one has}
 \begin{equation}\label{eq:pointwise-discrete-domination}
 |\Ces_{\mu,\beta}f(re^{i\theta})|
 \le K^\beta
 \left[
 T_{\mu,\beta}^{K}
 \left(\bigl(R(f)(r_je^{i\theta})\bigr)_{j\ge0}\right)
 \right]_n.
 \end{equation}
 \roundrev{Moreover,}
 \begin{equation}\label{eq:mixed-discrete-domination}
 \|\Ces_{\mu,\beta}f\|_{\AL_p^q(\omega_2)}
 \le K^\beta
 \left\|
 T_{\mu,\beta}^{K}
 \left(\bigl(R(f)(r_je^{i\theta})\bigr)_{j\ge0}\right)
 \right\|_{L^q(\ell^p(w_2))},
 \end{equation}
where, for a sequence \(a(\theta)=(a_n(\theta))_{n\ge0}\), we write
\begin{equation*}
 \|a\|_{L^q(\ell^p(w_i))}
 :=\left(\int_0^{2\pi}
 \left(\sum_{n\ge0}w_{i,n}|a_n(\theta)|^p\right)^{q/p}
 \frac{d\theta}{2\pi}\right)^{1/q}.
\end{equation*}
\end{lemma}

\begin{proof}
 Fix \(r\in I_n\), \(t\in I_k\), and \(\theta\in\mathbb R\). If
 \(k\le n\), then
 \[
 |1-rte^{i\theta}|\ge1-rt\ge1-t>K^{-(k+1)},
 \qquad rt\le t<r_{k+1}.
 \]
 It follows that
 \[
 \frac{|f(rte^{i\theta})|}{|1-rte^{i\theta}|^\beta}
 \le K^\beta K^{\beta k}R(f)(r_{k+1}e^{i\theta}).
 \]
 If \(k\ge n+1\), then
 \[
 |1-rte^{i\theta}|\ge1-rt\ge1-r>K^{-(n+1)},
 \qquad rt\le r<r_{n+1},
 \]
 and hence
 \[
 \frac{|f(rte^{i\theta})|}{|1-rte^{i\theta}|^\beta}
 \le K^\beta K^{\beta n}R(f)(r_{n+1}e^{i\theta}).
 \]
 Integrating these two estimates gives
 \begin{align*}
 |\Ces_{\mu,\beta}f(re^{i\theta})|
 &\le K^\beta\left[
 \sum_{k=0}^{n}K^{\beta k}m_k
 R(f)(r_{k+1}e^{i\theta})\right.\left.
 +K^{\beta n}M_{n+1}
 R(f)(r_{n+1}e^{i\theta})\right],
 \end{align*}
 which is \eqref{eq:pointwise-discrete-domination}. 
 
 For each fixed \(\theta\), we have
 \begin{align*}
 \int_0^1|\Ces_{\mu,\beta}f(re^{i\theta})|^p
 \omega_2(r)r\,dr&=\sum_{n\ge0}\int_{I_n}
 |\Ces_{\mu,\beta}f(re^{i\theta})|^p\omega_2(r)r\,dr\\
 &\le K^{\beta p}\sum_{n\ge0}w_{2,n}
 \left|
 \left[
 T_{\mu,\beta}^{K}
 \left(\bigl(R(f)(r_je^{i\theta})\bigr)_{j\ge0}\right)
 \right]_n
 \right|^p.
 \end{align*}
 Consequently,
 \[
 \begin{aligned}
 \|\Ces_{\mu,\beta}f\|_{\AL_p^q(\omega_2)}
 &=\left(
 \int_0^{2\pi}
 \left[
 \int_0^1 |\Ces_{\mu,\beta}f(re^{i\theta})|^p
 \omega_2(r)r\,dr
 \right]^{q/p}
 \frac{d\theta}{2\pi}
 \right)^{1/q}\\
 &\le K^\beta
 \left(
 \int_0^{2\pi}
 \left[
 \sum_{n\ge0}w_{2,n}
 \left|
 \left[
 T_{\mu,\beta}^{K}
 \left(\bigl(R(f)(r_je^{i\theta})\bigr)_{j\ge0}\right)
 \right]_n
 \right|^p
 \right]^{q/p}
 \frac{d\theta}{2\pi}
 \right)^{1/q}\\
 &=K^\beta
 \left\|
 T_{\mu,\beta}^{K}
 \left(\bigl(R(f)(r_je^{i\theta})\bigr)_{j\ge0}\right)
 \right\|_{L^q(\ell^p(w_2))},
 \end{aligned}
 \]
 which proves \eqref{eq:mixed-discrete-domination}.

\end{proof}

Since Lemma~\ref{thm:shifted-domination} holds for every \(K>1\), the
following lemma chooses \(K\) sufficiently large for the doubling
conditions of both weights and establishes the comparisons between
\(w_{i,n}\) and \(\widehat\omega_i(r_n)\), \(i=1,2\), needed below.

\medskip
\begin{lemma}\label{prop:common-grid}
Let \(\omega_1,\omega_2\in\Dclass\). There exist \(K\ge2\), integers
\(n_0\ge0\) and \(j_0\ge1\), and constants \(A,B>1\) such that, for the
corresponding grid \(r_n=1-K^{-n}\), \(I_n=[r_n,r_{n+1})\), one has

\begin{equation}\label{eq:adjacent-tail-ratio}
 A^{-j_0}\widehat\omega_i(r_n)
 \le\widehat\omega_i(r_{n+1})
 \le B^{-1}\widehat\omega_i(r_n)
\end{equation}
for
\(i=1,2\) and \(n\ge n_0\). Moreover,
\begin{equation}\label{eq:interval-tail-comparison}
 \frac{1-B^{-1}}{2}\widehat\omega_i(r_n)
 \le w_{i,n}\le(1-A^{-j_0})\widehat\omega_i(r_n),
\end{equation}
and
\begin{equation}\label{eq:discrete-tail-comparison}
 \frac12\widehat\omega_i(r_n)
 \le\sum_{j\ge n}w_{i,j}\le\widehat\omega_i(r_n)
\end{equation} for
\(i=1,2\) and \(n\ge n_0\). 
\end{lemma}

\begin{proof}
For \(i=1,2\), choose \(A_i>1\), \(B_i>1\), \(\kappa_i>1\), and
\(r_i^*<1\) as in \eqref{eq:D-plus}--\eqref{eq:D-minus}. Set
\[
 A=\max\{A_1,A_2\},\qquad B=\min\{B_1,B_2\},\qquad
 \kappa=\max\{\kappa_1,\kappa_2\},\qquad r_*=\max\{r_1^*,r_2^*\}.
\]
By the monotonicity of \(\widehat\omega_1\) and \(\widehat\omega_2\),
both estimates remain valid with these common constants.

Choose \(K\ge\max\{2,\kappa\}\). For the corresponding grid,
\[
 r_{n+1}=1-K^{-(n+1)}=1-\frac{1-r_n}{K}.
\]
Then choose \(n_0\) so that
\(r_{n_0}\ge\max\{1/2,r_*\}\), and choose \(j_0\ge1\) such that
\(2^{j_0}\ge K\). Thus \(r\ge1/2\) whenever \(n\ge n_0\) and
\(r\in I_n\).

Iterating \eqref{eq:D-plus} \(j_0\) times and using \(2^{j_0}\ge K\)
gives
\[
 \widehat\omega_i(r_n)
 \le A^{j_0}\widehat\omega_i\!\left(1-\frac{1-r_n}{2^{j_0}}\right)
 \le A^{j_0}\widehat\omega_i(r_{n+1}).
\]
Hence
\(\widehat\omega_i(r_{n+1})\ge
A^{-j_0}\widehat\omega_i(r_n)\). On the other hand, \(K\ge\kappa\)
implies
\[
 \widehat\omega_i(r_{n+1})
 \le\widehat\omega_i\!\left(1-\frac{1-r_n}{\kappa}\right)
 \le B^{-1}\widehat\omega_i(r_n).
\]
This proves \eqref{eq:adjacent-tail-ratio}.

It follows that
\[
 (1-B^{-1})\widehat\omega_i(r_n)
 \le\int_{r_n}^{r_{n+1}}\omega_i(r)\,dr
 \le(1-A^{-j_0})\widehat\omega_i(r_n).
\]
Since \(1/2\le r<1\) on \(I_n\) for \(n\ge n_0\),
\[
 \frac12\int_{r_n}^{r_{n+1}}\omega_i(r)\,dr
 \le w_{i,n}
 \le\int_{r_n}^{r_{n+1}}\omega_i(r)\,dr,
\]
which proves \eqref{eq:interval-tail-comparison}. Finally,
\(\sum_{j\ge n}w_{i,j}=\int_{r_n}^1\omega_i(r)r\,dr\), and
\(r_n\ge1/2\) gives \eqref{eq:discrete-tail-comparison}.
\end{proof}

\section{Boundedness of the Ces\`aro-type operator}
\label{sec:main-theorem}

In this section, we characterize the boundedness of
\(\Ces_{\mu,\beta}:\AL_p^q(\omega_1)\to\AL_p^q(\omega_2)\). Let \(K\) and
\(n_0\) be as in Lemma~\ref{prop:common-grid}. For every nonnegative sequence \(a=(a_n)_{n\ge0}\),
we consider the truncated sequence operator 
\begin{equation*}
 (\Tcut a)_n
 =K^{\beta n}M_{n+1}a_{n+1}
 +\sum_{k=n_0}^{n}K^{\beta k}m_k a_{k+1},
 \qquad n\ge n_0.
\end{equation*}
The following lemma is essential to proving boundedness, as it allows us to reduce the norm estimate for $\Ces_{\mu,\beta}$ to an estimate for $\Tcut$.

\medskip
\begin{lemma}\label{prop:finite-core}
For \(0<p,q<\infty\), \(\beta>0\), \(\omega_1,\omega_2\in\Dclass\),
and let \(\mu\) be a positive Borel measure on \([0,1)\). Let \(K\) and
\(n_0\) be as in Lemma~\ref{prop:common-grid}. Then, for every
\(f\in H(\mathbb D)\),
\begin{equation}\label{eq:finite-core-estimate}
 \|\Ces_{\mu,\beta}f\|_{\AL_p^q(\omega_2)}
 \lesssim K^\beta\left\|\Tcut\big((R(f)(r_je^{i\theta}))_{j\ge0}\big)\right\|_{L^q(\ell^p(w_2;n\ge n_0))}
 +\mu([0,1))\|f\|_{\AL_p^q(\omega_1)}.
\end{equation}
\end{lemma}

\begin{proof}
Fix \(\theta\). We consider the defining integral of \(\Ces_{\mu,\beta}\) in three cases. When
\(r\ge r_{n_0}\), we split it at \(t=r_{n_0}\) as
\[
 (\Ces_{\mu,\beta}f)(re^{i\theta})
 =\int_{[0,r_{n_0})}
 \frac{f(rte^{i\theta})}{(1-rte^{i\theta})^\beta}\,d\mu(t)
 +\int_{[r_{n_0},1)}
 \frac{f(rte^{i\theta})}{(1-rte^{i\theta})^\beta}\,d\mu(t),
\]
whereas for \(r<r_{n_0}\) we keep the defining integral unchanged. We estimate these
three cases separately.

\emph{Case 1: For \(r<r_{n_0}\).} We have
\(rt<r_{n_0}\) and \(1-rt>K^{-n_0}\) for every \(0\le t<1\). Hence
\[
 |\Ces_{\mu,\beta}f(re^{i\theta})|
 \le K^{\beta n_0}\mu([0,1))R(f)(r_{n_0}e^{i\theta}).
\]
Integrating over \(\bigcup_{n<n_0}I_n=[0,r_{n_0})\) gives
\begin{equation}\label{eq:finite-core-inner-output}
 \left(
 \sum_{n<n_0}\int_{I_n}|\Ces_{\mu,\beta}f(re^{i\theta})|^p
 \omega_2(r)r\,dr
 \right)^{1/p}
 \le K^{\beta n_0}\mu([0,1))
 \left(\sum_{n<n_0}w_{2,n}\right)^{1/p}R(f)(r_{n_0}e^{i\theta}).
\end{equation}

\emph{Case 2: For \(r\ge r_{n_0}\) and \(t<r_{n_0}\).}
For the part of the defining integral over \([0,r_{n_0})\), we have
\(rt<r_{n_0}\) and \(1-rt>K^{-n_0}\). Thus this part is bounded by
\(K^{\beta n_0}\mu([0,r_{n_0}))R(f)(r_{n_0}e^{i\theta})\), and its
\(L^p\)-quasinorm over \([r_{n_0},1)\) is at most
\begin{equation}\label{eq:finite-core-inner-input}
 K^{\beta n_0}\mu([0,r_{n_0}))
 \left(\sum_{n\ge n_0}w_{2,n}\right)^{1/p}R(f)(r_{n_0}e^{i\theta}).
\end{equation}

\emph{Case 3: For \(r,t\ge r_{n_0}\).} Consider the part
of the defining integral over \([r_{n_0},1)\), and set
\(\nu=\mu|_{[r_{n_0},1)}\). For \(n\ge n_0\),
\[
 \nu(I_k)=0\quad(k<n_0),\qquad
 \nu(I_k)=m_k\quad(k\ge n_0),\qquad
 \nu([r_{n+1},1))=M_{n+1}.
\]
Hence the shifted sequence operator associated with \(\nu\) is exactly
\(\Tcut\), and Lemma~\ref{thm:shifted-domination} gives
\[
 \left|
 \int_{[r_{n_0},1)}
 \frac{f(rte^{i\theta})}{(1-rte^{i\theta})^\beta}\,d\mu(t)
 \right|
 \le K^\beta
 \big[\Tcut((R(f)(r_je^{i\theta}))_{j\ge0})\big]_n.
\]

Combining the three cases, using the triangle inequality when \(p\ge1\)
and \((x+y)^p\le x^p+y^p\) when \(0<p<1\), yields
\[
\begin{aligned}
 \left(\int_0^1|\Ces_{\mu,\beta}f(re^{i\theta})|^p
 \omega_2(r)r\,dr\right)^{1/p}
 &\lesssim_p K^\beta
 \left\|\Tcut((R(f)(r_je^{i\theta}))_{j\ge0})\right\|_
 {\ell^p(w_2;n\ge n_0)}\\
 &\quad+K^{\beta n_0}\mu([0,1))
 \left(\sum_{n<n_0}w_{2,n}\right)^{1/p}R(f)(r_{n_0}e^{i\theta})\\
 &\quad+K^{\beta n_0}\mu([0,r_{n_0}))
 \left(\sum_{n\ge n_0}w_{2,n}\right)^{1/p}R(f)(r_{n_0}e^{i\theta}).
\end{aligned}
\]
For \(r\in I_j\), the radial maximal function is increasing along each
radius, and hence
\[
 w_{1,n_0}R(f)(r_{n_0}e^{i\theta})^p
 \le \sum_{j\ge0}w_{1,j}R(f)(r_je^{i\theta})^p
 \le\int_0^1R(f)(re^{i\theta})^p\omega_1(r)r\,dr.
\]
Therefore,
\begin{equation}\label{eq:radial-max-at-rn0}
 \|R(f)(r_{n_0}e^{i\theta})\|_{L^q(\mathbb T)}
 \le w_{1,n_0}^{-1/p}\left\|(R(f)(r_je^{i\theta}))_{j\ge0}\right\|_{L^q(\ell^p(w_1))}
 \lesssim w_{1,n_0}^{-1/p}\|f\|_{\AL_p^q(\omega_1)}.
\end{equation}
Finally, since
\[
 \sum_{n<n_0}w_{2,n}
 =\int_0^{r_{n_0}}\omega_2(r)r\,dr<\infty,
 \qquad
 \sum_{n\ge n_0}w_{2,n}
 =\int_{r_{n_0}}^1\omega_2(r)r\,dr<\infty,
\]
\(w_{1,n_0}>0\) by \eqref{eq:interval-tail-comparison}, and
\(\mu([0,r_{n_0}))\le\mu([0,1))\), it follows from
\eqref{eq:radial-max-at-rn0} that the last two terms on the right-hand side
of the estimate above are bounded in \(L^q(\mathbb T)\) by
\(C\,\mu([0,1))\|f\|_{\AL_p^q(\omega_1)}\). This proves
\eqref{eq:finite-core-estimate}. The implicit constant depends only on
\(p,q,\beta,K,n_0\) and $\omega_1,\omega_2$.
\end{proof}

For \(1/2\le r<1\), set
\begin{equation}\label{eq:E-general}
 E(r)
 :=\frac{\mu([r,1))}{(1-r)^\beta}
 \left(\frac{\widehat\omega_2(r)}
 {\widehat\omega_1(r)}\right)^{1/p}.
\end{equation}
Hence,
\[
 \Ctail(\mu;\omega_1,\omega_2)
 =\mu([0,1))+\sup_{1/2\le r<1}E(r).
\]
Next, we compare this continuous quantity with its discretized version.

\begin{lemma}\label{prop:continuous-discrete-tail}
For \(0<p<\infty\), \(\beta>0\),
 \(\omega_1,\omega_2\in\Dclass\), and let \(\mu\) be a positive Borel
 measure on \([0,1)\). Let \(K\) and \(n_0\) be as in
 Lemma~\ref{prop:common-grid}. Then
 \begin{equation}\label{eq:C-E-equivalence}
 \Ctail(\mu;\omega_1,\omega_2)
 \asymp\mu([0,1))+\sup_{n\ge n_0}E(r_n).
 \end{equation}
\end{lemma}

\begin{proof}
 By construction, \(r_n\ge1/2\) for \(n\ge n_0\). Hence
 \(\mu([0,1))+\sup_{n\ge n_0}E(r_n)
 \le \Ctail(\mu;\omega_1,\omega_2)\).

 Conversely, fix \(r\in[r_{n_0},1)\). Then there exists a unique
 \(n\ge n_0\) such that \(r\in I_n\). Hence
 \(\mu([r,1))\le M_n\) and
 \((1-r)^{-\beta}\le K^\beta K^{\beta n}\).

 Moreover, by monotonicity and Lemma~\ref{prop:common-grid},
 \[
 \widehat\omega_2(r)\le\widehat\omega_2(r_n),\qquad
 \widehat\omega_1(r)\ge\widehat\omega_1(r_{n+1})
 \ge A^{-j_0}\widehat\omega_1(r_n).
 \]
 Hence
 \[
 \frac{\widehat\omega_2(r)}{\widehat\omega_1(r)}
 \le A^{j_0}
 \frac{\widehat\omega_2(r_n)}{\widehat\omega_1(r_n)}.
 \]
 It follows that
 \[
  E(r)\le K^\beta A^{j_0/p}E(r_n)
  \le K^\beta A^{j_0/p}\sup_{j\ge n_0}E(r_j).
 \]

 For \(r\in[1/2,r_{n_0})\), we have
 \[
  E(r)
  \le K^{\beta n_0}
  \left(\frac{\widehat\omega_2(1/2)}
  {\widehat\omega_1(r_{n_0})}\right)^{1/p}
  \mu([0,1))
  \lesssim \mu([0,1)).
 \]
 Therefore,
 \(\Ctail(\mu;\omega_1,\omega_2)
 \lesssim \mu([0,1))+\sup_{n\ge n_0}E(r_n)\), which proves
 \eqref{eq:C-E-equivalence}. The implicit constants depend only on
 \(p,\beta,K,n_0\) and $\omega_1,\omega_2$.
\end{proof}

We need the following lemma for the proof of the main theorem.

\begin{lemma}\label{prop:truncated-sequence-bound}
 For \(0<p<\infty\), \(\beta>0\),
 \(\omega_1,\omega_2\in\Dclass\), and let \(\mu\) be a positive Borel
 measure on \([0,1)\). Let \(K\) and \(n_0\) be as in
 Lemma~\ref{prop:common-grid}. For every nonnegative
 sequence \(a=(a_n)_{n\ge0}\),
 \begin{equation}\label{eq:truncated-sequence-bound}
 \left\|\Tcut a\right\|_{\ell^p(w_2;n\ge n_0)}
 \lesssim
 \left(\sup_{j\ge n_0}E(r_j)\right)
 \left(\sum_{k\ge n_0}w_{1,k+1}a_{k+1}^p\right)^{1/p}.
 \end{equation}
\end{lemma}

\begin{proof}
 Set \(x_k=w_{1,k+1}^{1/p}a_{k+1}\) for \(k\ge n_0\), and
 \(x_k=0\) for \(k<n_0\). By Lemma~\ref{prop:common-grid},
 \begin{equation}\label{eq:grid-weights-for-direct-bound}
 w_{2,n}\lesssim\widehat\omega_2(r_n),
 \qquad
 w_{1,k+1}\gtrsim\widehat\omega_1(r_k),
 \qquad n,k\ge n_0.
 \end{equation}
 Since \(M_{n+1}\le M_n\), the definition of \(E(r_n)\) and
 \eqref{eq:grid-weights-for-direct-bound} give
 \[
 w_{2,n}^{1/p}K^{\beta n}M_{n+1}a_{n+1}
 \lesssim \left(\sup_{j\ge n_0}E(r_j)\right)x_n.
 \]
 Similarly, \(m_k\le M_k\) implies
 \[
 \begin{aligned}
 w_{2,n}^{1/p}\sum_{k=n_0}^{n}K^{\beta k}m_ka_{k+1}
 &\lesssim \left(\sup_{j\ge n_0}E(r_j)\right)
 \sum_{k=n_0}^{n}
 \left(\frac{\widehat\omega_2(r_n)}
 {\widehat\omega_2(r_k)}\right)^{1/p}x_k.
 \end{aligned}
 \]
 Iterating the upper estimate in \eqref{eq:adjacent-tail-ratio}, we have
 \(\widehat\omega_2(r_n)\le B^{-(n-k)}\widehat\omega_2(r_k)\) for
 \(n\ge k\ge n_0\). Thus
 \begin{equation}\label{eq:geometric-convolution-domination}
 w_{2,n}^{1/p}(\Tcut a)_n
 \lesssim
 \left(\sup_{j\ge n_0}E(r_j)\right)
 \sum_{k=n_0}^{n}B^{-(n-k)/p}x_k,
 \qquad n\ge n_0.
 \end{equation}

 Suppose first that \(p\ge1\). Viewing the sum in
 \eqref{eq:geometric-convolution-domination} as the convolution of \(x\)
 with the \(\ell^1\)-sequence \(\{B^{-j/p}\}_{j\ge0}\), the discrete Young
 inequality gives
 \[
 \left[\sum_{n\ge n_0}\left(\sum_{k=n_0}^{n}B^{-(n-k)/p}x_k\right)^p\right]^{1/p}
 \le \left(\sum_{j\ge0}B^{-j/p}\right)
 \left(\sum_{k\ge n_0}x_k^p\right)^{1/p}
 \lesssim \left(\sum_{k\ge n_0}x_k^p\right)^{1/p}.
 \]
 If \(0<p<1\), then subadditivity yields
 \[
 \sum_{n\ge n_0}\left(\sum_{k=n_0}^{n}B^{-(n-k)/p}x_k\right)^p
 \le \sum_{n\ge n_0}\sum_{k=n_0}^{n}B^{-(n-k)}x_k^p
 =\frac{1}{1-B^{-1}}\sum_{k\ge n_0}x_k^p.
 \]
 Combining either estimate with
 \eqref{eq:geometric-convolution-domination} proves
 \eqref{eq:truncated-sequence-bound}. The implicit constant depends only on
 \(p,A,B,j_0\).
\end{proof}

We now prove Theorem~\ref{thm:main-two-weight}.
\begin{proof}[Proof of Theorem~\ref{thm:main-two-weight}]
 Suppose first that
 \(\Ces_{\mu,\beta}:\AL_p^q(\omega_1)\to\AL_p^q(\omega_2)\)
 is bounded. Let \(\beta_0=\beta_0(\omega_1)\) be as in
 Lemma~\ref{lem:cited-high-kernel}, and choose an integer
 \(\gamma>\frac1q+\frac{\beta_0}{p}\).
 For \(0\le a<1\), apply Lemma~\ref{lem:cited-high-kernel} to
 \(f_a(z)=(1-az)^{-\gamma}\), \(z\in\Disk\). This gives
 \[
 \|f_a\|_{\AL_p^q(\omega_1)}
 \lesssim\what\omega_1(a)^{1/p}(1-a)^{1/q-\gamma}.
 \]
 Hence, Lemma~\ref{lem:cited-point-evaluation} and the boundedness of
 \(\Ces_{\mu,\beta}\) imply
 \begin{equation}\label{eq:general-necessity-upper}
 |(\Ces_{\mu,\beta}f_a)(a)|
 \lesssim \|\Ces_{\mu,\beta}\|
 \frac{\|f_a\|_{\AL_p^q(\omega_1)}}
 {\what\omega_2(a)^{1/p}(1-a)^{1/q}}
 \lesssim \|\Ces_{\mu,\beta}\|
 \left(\frac{\what\omega_1(a)}{\what\omega_2(a)}\right)^{1/p}
 (1-a)^{-\gamma}.
 \end{equation}
 For \(t\in[a,1)\), we have \(1-at\le2(1-a)\) and
 \(1-a^2t\le3(1-a)\). Since \(\mu\) is positive,
 \begin{equation}\label{eq:general-necessity-lower}
 (\Ces_{\mu,\beta}f_a)(a)
 \ge2^{-\beta}3^{-\gamma}\mu([a,1))(1-a)^{-\gamma-\beta}.
 \end{equation}
 Comparing \eqref{eq:general-necessity-upper} and
 \eqref{eq:general-necessity-lower}, we obtain
 \begin{equation}\label{eq:general-necessary-tail}
 \mu([a,1))\lesssim
 \|\Ces_{\mu,\beta}\|(1-a)^\beta
 \left(\frac{\what\omega_1(a)}{\what\omega_2(a)}\right)^{1/p},
 \qquad 0\le a<1.
 \end{equation}
 Taking \(a=0\) gives \(\mu([0,1))<\infty\), and hence
 \(\Ctail(\mu;\omega_1,\omega_2)\lesssim\|\Ces_{\mu,\beta}\|\).

 Conversely, assume that
 \(\Ctail(\mu;\omega_1,\omega_2)<\infty\). By
 Lemma~\ref{prop:continuous-discrete-tail},
 \[
  \sup_{n\ge n_0}E(r_n)
  \lesssim\Ctail(\mu;\omega_1,\omega_2).
 \]
 For each fixed
 \(\theta\), Lemma~\ref{prop:truncated-sequence-bound}, applied to
 the nonnegative sequence \(\bigl(R(f)(r_je^{i\theta})\bigr)_{j\ge0}\), gives
 \[
 \begin{aligned}
 \left\|
 \Tcut\left(\bigl(R(f)(r_je^{i\theta})\bigr)_{j\ge0}\right)
 \right\|_{\ell^p(w_2;n\ge n_0)}
 &\lesssim
 \Ctail(\mu;\omega_1,\omega_2)
 \left(\sum_{k\ge n_0}w_{1,k+1}R(f)(r_{k+1}e^{i\theta})^p\right)^{1/p}\\
 &\le
 \Ctail(\mu;\omega_1,\omega_2)
 \left(\int_{r_{n_0+1}}^1
 R(f)(re^{i\theta})^p\omega_1(r)r\,dr\right)^{1/p}.
 \end{aligned}
 \]
 Here the last inequality follows directly from the monotonicity of
 \(R(f)\) along each radius. Indeed, if \(r\in I_{k+1}\), then
 \(R(f)(r_{k+1}e^{i\theta})\le R(f)(re^{i\theta})\). Therefore,
 \[
 \sum_{k\ge n_0}w_{1,k+1}R(f)(r_{k+1}e^{i\theta})^p
 \le\int_{r_{n_0+1}}^1
 R(f)(re^{i\theta})^p\omega_1(r)r\,dr.
 \]
 Taking the
 \(L^q(\mathbb T)\)-quasinorm and using
 Lemma~\ref{lem:cited-maximal}, we obtain
 \[
 \begin{aligned}
 \left\|
 \Tcut\left(\bigl(R(f)(r_je^{i\theta})\bigr)_{j\ge0}\right)
 \right\|_{L^q(\ell^p(w_2;n\ge n_0))}
 &\lesssim
 \Ctail(\mu;\omega_1,\omega_2)\|R(f)\|_{L_p^q(\omega_1)}\\
 &\lesssim
 \Ctail(\mu;\omega_1,\omega_2)\|f\|_{\AL_p^q(\omega_1)}.
 \end{aligned}
 \]
 Together with Lemma~\ref{prop:finite-core}, this gives
 \[
 \|\Ces_{\mu,\beta}f\|_{\AL_p^q(\omega_2)}
 \lesssim
 \Ctail(\mu;\omega_1,\omega_2)
 \|f\|_{\AL_p^q(\omega_1)}.
 \]
 Thus,
 \(\Ces_{\mu,\beta}:\AL_p^q(\omega_1)\to\AL_p^q(\omega_2)\)
 is bounded, and the norm comparison follows, completing the proof.
\end{proof}

\section{Compactness and essential norm of the Ces\`aro-type operator}
\label{sec:essential-norm}

This section deals with the compactness and essential norm of
\(\Ces_{\mu,\beta}:\AL_p^q(\omega_1)\to\AL_p^q(\omega_2)\).
We first prove two lemmas.

\begin{lemma}\label{lem:normalized-evaluations}
 For \(0<p,q<\infty\) and \(\omega\in\Dhat\). Let \(1/2\le a<1\), define
 \(E_a:\AL_p^q(\omega)\to\mathbb C\) by
 \[
  E_a(f)=\what\omega(a)^{1/p}(1-a)^{1/q}f(a),
  \qquad f\in\AL_p^q(\omega).
 \]
 Then \(\sup_{1/2\le a<1}\|E_a\|<\infty\), and \(E_a\to0\) uniformly on
 every compact subset of \(\AL_p^q(\omega)\) as \(a\to1^-\).
\end{lemma}

\begin{proof}
 The uniform boundedness follows from
 Lemma~\ref{lem:cited-point-evaluation}. For a polynomial \(P\),
 \(|E_a(P)|\le \|P\|_{H^\infty}\what\omega(a)^{1/p}(1-a)^{1/q}\to0\).
 By \cite[Theorem~1.2(iii)]{AguilarHernandezEtAl2026}, analytic
 polynomials are dense in \(\AL_p^q(\omega)\), and hence the preceding
 estimate together with the uniform bound for \(E_a\) gives
 \(E_a(f)\to0\) for every \(f\in\AL_p^q(\omega)\). Let
 \(\mathcal K\subset\AL_p^q(\omega)\) be compact and put
 \(M=\sup_{1/2\le a<1}\|E_a\|\). Given \(\varepsilon>0\), choose
 \(f_1,\ldots,f_N\in\mathcal K\) such that for every \(f\in\mathcal K\) there
 is some \(j\) with
 \(\|f-f_j\|_{\AL_p^q(\omega)}<\varepsilon/[2(M+1)]\). Since \(E_a(f_j)\to0\) for each \(j\), there exists
 \(a_0\in(1/2,1)\) such that
 \[
 \max_{1\le j\le N}|E_a(f_j)|<\frac{\varepsilon}{2},
 \qquad a_0<a<1.
 \]
 Hence, for \(a_0<a<1\),
 \[
 \sup_{f\in\mathcal K}|E_a(f)|
 \le \max_{1\le j\le N}|E_a(f_j)|
 +\frac{M\varepsilon}{2(M+1)}<\varepsilon.
 \]
 Thus \(E_a\to0\) uniformly on \(\mathcal K\).
\end{proof}

\begin{lemma}\label{lem:compact-truncation}
 For \(0<p,q<\infty\), \(\beta>0\), \(\omega_1\in\Dhat\), and let
 \(\omega_2\) be a radial weight. Suppose that \(\mu\) be a finite positive
 Borel measure on \([0,1)\). If \(0<\tau<1\), put
 \(\mu_\tau=\mu|_{[0,\tau)}\). Then
 \(\Ces_{\mu_\tau,\beta}:\AL_p^q(\omega_1)\to\AL_p^q(\omega_2)\) is compact.
\end{lemma}

\begin{proof}
 Let \((f_j)\) be bounded in \(\AL_p^q(\omega_1)\). By
 Lemma~\ref{lem:cited-point-evaluation} and Montel's theorem, a subsequence,
 still denoted by \((f_j)\), converges uniformly on compact subsets of
 \(\Disk\) to an analytic function \(f\). Fatou's lemma gives
 \[
 \|f\|_{\AL_p^q(\omega_1)}^q
 \le \liminf_{j\to\infty}\|f_j\|_{\AL_p^q(\omega_1)}^q<\infty,
 \]
 so \(f\in\AL_p^q(\omega_1)\). By the definition of \(\mu_\tau\), for \(z\in\Disk\),
 \[
 \begin{aligned}
 |\Ces_{\mu_\tau,\beta}(f_j-f)(z)|
 &\le
 \int_{[0,\tau)}
 \frac{|f_j(tz)-f(tz)|}{|1-tz|^\beta}\,d\mu(t)\\
 &\le
 \frac{\mu([0,\tau))}{(1-\tau)^\beta}
 \sup_{|w|\le\tau}|f_j(w)-f(w)|.
 \end{aligned}
 \]
 Hence
 \[
 \begin{aligned}
 \|\Ces_{\mu_\tau,\beta}(f_j-f)\|_{\AL_p^q(\omega_2)}
 &\le
 \left(\int_0^1\omega_2(r)r\,dr\right)^{1/p}
 \sup_{z\in\Disk}|\Ces_{\mu_\tau,\beta}(f_j-f)(z)|\\
 &\le
 \left(\int_0^1\omega_2(r)r\,dr\right)^{1/p}
 \frac{\mu([0,\tau))}{(1-\tau)^\beta}
 \sup_{|w|\le \tau}|f_j(w)-f(w)|
 \longrightarrow0.
 \end{aligned}
 \]
 Hence \(\Ces_{\mu_\tau,\beta}\) is compact.
\end{proof}

We now prove Theorem~\ref{thm:intro-essential-norm}.

\begin{proof}[Proof of Theorem~\ref{thm:intro-essential-norm}]
 For \(0<\tau<1\), write \(\mu_\tau=\mu|_{[0,\tau)}\) and
 \(\nu_\tau=\mu|_{[\tau,1)}\). By Lemma~\ref{lem:compact-truncation},
 \(\Ces_{\mu_\tau,\beta}\) is compact, and hence
 \(\|\Ces_{\mu,\beta}\|_e\le\|\Ces_{\nu_\tau,\beta}\|
 \lesssim\Ctail(\nu_\tau;\omega_1,\omega_2)\).
 Fix \(1/2<\rho<1\) and take \(\tau>\rho\). For \(r\ge\rho\),
 \[
 \nu_\tau([r,1))
 =\mu([\max\{r,\tau\},1))
 \le \mu([r,1)).
 \]
 Hence
 \[
 \sup_{\rho\le r<1}
 \frac{\nu_\tau([r,1))}{(1-r)^\beta}
 \left(\frac{\what\omega_2(r)}{\what\omega_1(r)}\right)^{1/p}
 \le
 \sup_{\rho\le r<1}
 \frac{\mu([r,1))}{(1-r)^\beta}
 \left(\frac{\what\omega_2(r)}{\what\omega_1(r)}\right)^{1/p},
 \]
 whereas
 \[
 \sup_{1/2\le r<\rho}
 \frac{\nu_\tau([r,1))}{(1-r)^\beta}
 \left(\frac{\what\omega_2(r)}{\what\omega_1(r)}\right)^{1/p}
 \le \mu([\tau,1))
 \sup_{1/2\le r<\rho}\frac1{(1-r)^\beta}
 \left(\frac{\what\omega_2(r)}{\what\omega_1(r)}\right)^{1/p}
 \longrightarrow0.
 \]
 Also \(\nu_\tau([0,1))=\mu([\tau,1))\to0\). It follows that
 \[
 \limsup_{\tau\to1^-}\Ctail(\nu_\tau;\omega_1,\omega_2)
 \le \sup_{\rho\le r<1}
 \frac{\mu([r,1))}{(1-r)^\beta}
 \left(\frac{\what\omega_2(r)}{\what\omega_1(r)}\right)^{1/p}.
 \]
 Letting \(\rho\to1^-\) proves the upper estimate in
 \eqref{eq:intro-essential-norm}.

 For the reverse estimate, choose
 \(\gamma>1/q+\beta_0(\omega_1)/p\), where \(\beta_0(\omega_1)\) is as in
 Lemma~\ref{lem:cited-high-kernel}, and define
 \[
 g_a(z)=
 \frac{(1-a)^{\gamma-1/q}}
 {\what\omega_1(a)^{1/p}(1-az)^\gamma},
 \qquad \frac12\le a<1.
 \]
 Lemma~\ref{lem:cited-high-kernel} gives
 \(\sup_a\|g_a\|_{\AL_p^q(\omega_1)}<\infty\). Set
 \(E_a(h)=\what\omega_2(a)^{1/p}(1-a)^{1/q}h(a)\). If
 \(\mathcal K:\AL_p^q(\omega_1)\to\AL_p^q(\omega_2)\) is compact, then
 Lemma~\ref{lem:normalized-evaluations}, applied to the compact set
 \(\overline{\{\mathcal K g_a:1/2\le a<1\}}\), gives
 \(E_a(\mathcal K g_a)\to0\).
 For \(t\in[a,1)\), we have \(1-at\le2(1-a)\) and
 \(1-a^2t\le3(1-a)\). Positivity of \(\mu\) therefore gives
 \[
 E_a(\Ces_{\mu,\beta}g_a)
 \gtrsim
 \frac{\mu([a,1))}{(1-a)^\beta}
 \left(\frac{\what\omega_2(a)}{\what\omega_1(a)}\right)^{1/p}.
 \]
 Since \(E_a(\mathcal K g_a)\to0\), while
 \(\sup_a\|E_a\|<\infty\) and
 \(\sup_a\|g_a\|_{\AL_p^q(\omega_1)}<\infty\),
 \[
 \begin{aligned}
 \frac{\mu([a,1))}{(1-a)^\beta}
 \left(\frac{\what\omega_2(a)}{\what\omega_1(a)}\right)^{1/p}
 &\lesssim |E_a((\Ces_{\mu,\beta}-\mathcal K)g_a)|
      +|E_a(\mathcal K g_a)|\\
 &\le \|E_a\|\,\|\Ces_{\mu,\beta}-\mathcal K\|\,
      \|g_a\|_{\AL_p^q(\omega_1)}+|E_a(\mathcal K g_a)|.
 \end{aligned}
 \]
 Taking the upper limit as \(a\to1^-\) gives
 \[
 \limsup_{a\to1^-}\frac{\mu([a,1))}{(1-a)^\beta}
 \left(\frac{\what\omega_2(a)}{\what\omega_1(a)}\right)^{1/p}
 \lesssim \|\Ces_{\mu,\beta}-\mathcal K\|.
 \]
 Hence
 \[
 \limsup_{a\to1^-}\frac{\mu([a,1))}{(1-a)^\beta}
 \left(\frac{\what\omega_2(a)}{\what\omega_1(a)}\right)^{1/p}
 \lesssim\inf_{\mathcal K\ \mathrm{compact}}\|\Ces_{\mu,\beta}-\mathcal K\|
 =\|\Ces_{\mu,\beta}\|_e.
 \]
 This proves the reverse estimate in \eqref{eq:intro-essential-norm}.

 If \eqref{eq:intro-compactness-condition} holds, then
 \[
 \|\Ces_{\mu,\beta}-\Ces_{\mu_\tau,\beta}\|
 =\|\Ces_{\nu_\tau,\beta}\|
 \lesssim \Ctail(\nu_\tau;\omega_1,\omega_2)\longrightarrow0,
 \qquad \tau\to1^-.
 \]
 Since every \(\Ces_{\mu_\tau,\beta}\) is compact,
 \(\Ces_{\mu,\beta}\) is compact. Conversely, if
 \(\Ces_{\mu,\beta}\) is compact, then \(\|\Ces_{\mu,\beta}\|_e=0\), and
 \[
 0\le\limsup_{a\to1^-}\frac{\mu([a,1))}{(1-a)^\beta}
 \left(\frac{\what\omega_2(a)}{\what\omega_1(a)}\right)^{1/p}
 \lesssim\|\Ces_{\mu,\beta}\|_e=0.
 \]
 Thus \eqref{eq:intro-compactness-condition} follows.

\end{proof}

\section{Further remarks}\label{sec:examples-scope}

This section presents three applications of the main theorems.

\subsection{Standard weights}

We first discuss the standard weight case, where the conditions in Theorems~\ref{thm:main-two-weight} and~\ref{thm:intro-essential-norm} become simpler. This case further allows a direct comparison with the results obtained by Galanopoulos, Siskakis, Zhao
\cite{GalanopoulosSiskakisZhao2025} and Blasco, Mas
\cite{BlascoMas2026}. 

For \(\alpha>-1\), set \(\omega_\alpha(r)=(\alpha+1)(1-r^2)^\alpha\) and write
\(\AL_p^q(\alpha)=\AL_p^q(\omega_\alpha)\). When \(p=q\), Fubini's theorem yields
\[
 \AL_p^p(\alpha)=A_\alpha^p
 =H\!\left(p,p,\frac{\alpha+1}{p}\right).
\]
 As discussed in the introduction, Galanopoulos, Siskakis, and Zhao
\cite{GalanopoulosSiskakisZhao2025} obtained the boundedness result for
\(\Ces_{\mu,\beta}:A_\alpha^p\to A_\alpha^p\) with \(p\ge1\).  Blasco and Mas
\cite{BlascoMas2026} obtained the corresponding characterization for
\(p\ge1\) with \(s>0\) and for  \(0<p<1\), they additionally assumed
\(\alpha_2>\alpha_1\), equivalently \(\beta>s\).
The following corollary removes these restrictions:  it allows arbitrary \(\alpha_1,\alpha_2>-1\),
covers the full range \(0<p<\infty\), and also includes \(s\le0\). It also gives the
essential norm and compactness.  Note that $
\what\omega_\alpha(r)\asymp(1-r)^{\alpha+1}.
$ Thus, applying
Theorems~\ref{thm:main-two-weight} and~\ref{thm:intro-essential-norm}
to standard weights gives the following corollary.

\begin{corollary}\label{cor:standard-weights}
 For \(0<p,q<\infty\), \(\alpha_1,\alpha_2>-1\), and \(\beta>0\).   Put
 $
 s=\beta+\frac{\alpha_1-\alpha_2}{p}.
 $
 Let \(\mu\) be a positive Borel measure on \([0,1)\). Then
 \(\Ces_{\mu,\beta}:\AL_p^q(\alpha_1)\to\AL_p^q(\alpha_2)\) is bounded
 if and only if
 \[
 \begin{cases}
 \displaystyle
 \sup_{0\le r<1}\frac{\mu([r,1))}{(1-r)^s}<\infty,&s>0,\\[3mm]
 \mu([0,1))<\infty,&s\le0.
 \end{cases}
 \]
 Its norm is comparable to the corresponding quantity above. If the operator
 is bounded, then
 \[
 \bigl\|\Ces_{\mu,\beta}\bigr\|_e
 \asymp
 \limsup_{r\to1^-}\frac{\mu([r,1))}{(1-r)^s}.
 \]
 If \(s\le0\), the operator is compact. If \(s>0\), the operator is compact
 if and only if
 \[
 \lim_{r\to1^-}\frac{\mu([r,1))}{(1-r)^s}=0.
 \]
\end{corollary}

\subsection{Analytic weighted tent spaces}

We next recall the analytic weighted tent spaces. Let
\(dA(z)=dx\,dy/\pi\) be the normalized area measure on \(\Disk\). For
\(0<p,q<\infty\), \(\zeta\in\Torus\), and a fixed aperture \(\gamma>1\), set
\[
 \Gamma_\gamma(\zeta)
 =\{z\in\Disk:|z-\zeta|<\gamma(1-|z|)\}.
\]
For a radial weight \(\omega\), the analytic weighted tent space
\(\AT_p^q(\omega)\) consists of the functions \(f\in H(\mathbb D)\) such that
\[
 \|f\|_{\AT_p^q(\omega)}
 =\left(\int_{\Torus}
 \left[\int_{\Gamma_\gamma(\zeta)}|f(z)|^p\omega(z)
 \frac{dA(z)}{1-|z|}\right]^{q/p}|d\zeta|\right)^{1/q}<\infty.
\]
Tent spaces were introduced by Coifman, Meyer, and Stein
\cite{CoifmanMeyerStein1985}. In the weighted analytic setting,
Aguilar-Hern\'andez et al. show that
\begin{equation}\label{eq:example-AL-AT}
 \AL_p^q(\omega)=\AT_p^q(\omega),\qquad \omega\in\Dhat,
\end{equation}
with equivalent quasinorms in \cite[Theorem~1.3]{AguilarHernandezEtAl2026}. By \eqref{eq:example-AL-AT},
Theorems~\ref{thm:main-two-weight} and~\ref{thm:intro-essential-norm}
give the following corollary.

\begin{corollary}\label{cor:tent-spaces}
For \(0<p,q<\infty\), \(\beta>0\), \(\omega_1,\omega_2\in\Dclass\), and let
\(\mu\) be a positive Borel measure on \([0,1)\). Then
$
 \Ces_{\mu,\beta}:\AT_p^q(\omega_1)\longrightarrow\AT_p^q(\omega_2)
$
is bounded if and only if
\[
 \mu([0,1))+\sup_{1/2\le r<1}
 \frac{\mu([r,1))}{(1-r)^\beta}
 \left(\frac{\widehat\omega_2(r)}{\widehat\omega_1(r)}\right)^{1/p}
 <\infty.
\]
Moreover,
\[
 \bigl\|\Ces_{\mu,\beta}\bigr\|_{\AT_p^q(\omega_1)\to\AT_p^q(\omega_2)}
 \asymp
 \mu([0,1))+\sup_{1/2\le r<1}
 \frac{\mu([r,1))}{(1-r)^\beta}
 \left(\frac{\widehat\omega_2(r)}{\widehat\omega_1(r)}\right)^{1/p}.
\]
If the operator is bounded, then
\[
 \bigl\|\Ces_{\mu,\beta}\bigr\|_e
 \asymp
 \limsup_{r\to1^-}
 \frac{\mu([r,1))}{(1-r)^\beta}
 \left(\frac{\widehat\omega_2(r)}{\widehat\omega_1(r)}\right)^{1/p},
\]
and it is compact if and only if
\[
 \lim_{r\to1^-}
 \frac{\mu([r,1))}{(1-r)^\beta}
 \left(\frac{\widehat\omega_2(r)}{\widehat\omega_1(r)}\right)^{1/p}=0.
\]
\end{corollary}

\subsection{Logarithmic-type weights}

For \(i=1,2\), let
\[
 \omega_i(r)=(1-r)^{\alpha_i}
 \left(\log\frac e{1-r}\right)^{\gamma_i},
 \qquad \alpha_i>-1,\quad\gamma_i\in\mathbb R.
\]
 It is known that \(\omega_i\in\Dclass\) for each \(i=1,2\). See \cite{AguilarHernandezEtAl2026}. If \(\gamma_i\ne0\), then
\(\omega_i\) is not a standard weight. 

\begin{corollary}\label{cor:logarithmic-weights}
 For \(0<p,q<\infty\), \(\beta>0\), and let \(\mu\) be a positive Borel measure on
 \([0,1)\). Then
 \(\Ces_{\mu,\beta}:\AL_p^q(\omega_1)\to\AL_p^q(\omega_2)\) is bounded
 if and only if
 \[
 \mu([0,1))+\sup_{1/2\le r<1}
 \frac{\mu([r,1))}
 {(1-r)^{\beta+(\alpha_1-\alpha_2)/p}}
 \left(\log\frac e{1-r}\right)^{(\gamma_2-\gamma_1)/p}<\infty.
 \]
 If the operator is bounded, then
 \[
 \bigl\|\Ces_{\mu,\beta}\bigr\|_e
 \asymp
 \limsup_{r\to1^-}
 \frac{\mu([r,1))}
 {(1-r)^{\beta+(\alpha_1-\alpha_2)/p}}
 \left(\log\frac e{1-r}\right)^{(\gamma_2-\gamma_1)/p}.
 \]
 The operator is compact if and only if the limsup tends to zero as \(r\to1^-\).
\end{corollary}

\begin{proof}
 
 It is only to show that
 \[
 \widehat\omega_i(r)\asymp
 (1-r)^{\alpha_i+1}
 \left(\log\frac e{1-r}\right)^{\gamma_i}, \quad i=1,2.
 \]
 Indeed, split \([r,1)\) into
 \[
 I_j=
 \left[1-2^{-j}(1-r),\,1-2^{-j-1}(1-r)\right),
 \qquad j\ge0.
 \]
 Then, it is easy to see that
 \[
 \int_{I_j}\omega_i(t)\,dt
 \asymp
 (1-r)^{\alpha_i+1}2^{-j(\alpha_i+1)}
 \left(\log\frac e{1-r}+j\right)^{\gamma_i}.
 \]
 Thus, we have
 \[
 \widehat\omega_i(r)
 \ge \int_{I_0}\omega_i(t)\,dt
 \asymp
 (1-r)^{\alpha_i+1}
 \left(\log\frac e{1-r}\right)^{\gamma_i},
 \]
On the other hand,
 \[
 \begin{aligned}
 	\widehat\omega_i(r)
 	&\lesssim
 	(1-r)^{\alpha_i+1}
 	\sum_{j\ge0}2^{-j(\alpha_i+1)}
 	\left(\log\frac e{1-r}+j\right)^{\gamma_i} \\
 	&\lesssim
 	(1-r)^{\alpha_i+1}
 	\left(\log\frac e{1-r}\right)^{\gamma_i}
 	\sum_{j\ge0}2^{-j(\alpha_i+1)}(1+j)^{|\gamma_i|} \\
 	&\lesssim
 	(1-r)^{\alpha_i+1}
 	\left(\log\frac e{1-r}\right)^{\gamma_i},
 \end{aligned}
 \]
Consequently,
 \[
 \frac{1}{(1-r)^\beta}
 \left(\frac{\widehat\omega_2(r)}{\widehat\omega_1(r)}\right)^{1/p}
 \asymp
 \frac{1}{(1-r)^{\beta+(\alpha_1-\alpha_2)/p}}
 \left(\log\frac e{1-r}\right)^{(\gamma_2-\gamma_1)/p}.
 \]
 The conclusions follow from Theorems~\ref{thm:main-two-weight}
 and~\ref{thm:intro-essential-norm}.
\end{proof}

\noindent\textbf{Data availability}
No data were used for the research described in this article.

\end{document}